\documentclass[a4paper,10pt]{amsart}
\usepackage{amsmath, amssymb, hyperref, color}

\theoremstyle{plain}
\newtheorem{theorem}{\bf Theorem}[section]

\newtheorem{lemma}[theorem]{\bf Lemma}
\newtheorem{corollary}[theorem]{\bf Corollary}
\newtheorem{conjecture}[theorem]{\bf Conjecture}
\theoremstyle{definition}

\newtheorem{remark}[theorem]{\bf Remark}

\newcommand{\N}{\mathbb N}
\newcommand{\Z}{\mathbb Z}
\newcommand{\R}{\mathbb R}

\DeclareMathOperator{\ord}{ord}

\DeclareMathOperator{\supp}{supp}

\numberwithin{equation}{section}

\begin{document}
\title{On a zero-sum problem arising from factorization theory}

\address{Institut f\"ur Mathematik und wissenschaftliches Rechnen, Karl-Franzens-Universit\"at Graz, NAWI Graz, Heinrichstra{\ss}e 36, 8010 Graz, Austria}
\email{aqsa.bashir@uni-graz.at,alfred.geroldinger@uni-graz.at,qinghai.zhong@uni-graz.at}
\urladdr{https://imsc.uni-graz.at/geroldinger, https://imsc.uni-graz.at/zhong/}
\author{Aqsa Bashir and Alfred Geroldinger and Qinghai Zhong}

\thanks{This work was supported by the Austrian Science Fund FWF, Project Numbers  W1230 and P33499-N}

\begin{abstract}
We study a zero-sum problem dealing with minimal zero-sum sequences of maximal length over finite abelian groups. A positive answer to this problem yields a structural description of sets of lengths with maximal elasticity in transfer Krull monoids over finite abelian groups.
\end{abstract}

\maketitle

\section{Introduction}\label{1}

Let $G$ be an additively written, finite abelian group and $G_0 \subset G$ be a subset. By a sequence $S = g_1 \ldots g_{\ell}$ over $G_0$, we mean a finite sequence of terms from $G_0$, where the order is disregarded and repetition is allowed. We say that $S$ has sum zero if $g_1 + \ldots + g_{\ell}=0$ and that $S$ is a minimal zero-sum sequence if no proper subsum equals zero (i.e., $\sum_{i \in I} g_i \ne 0$ for all $\emptyset \ne I \subsetneq [1, \ell]$). The set of all zero-sum sequences is a (multiplicative) monoid with concatenation of sequences as operation. The empty sequence is the identity element of this monoid and the minimal zero-sum sequences are the irreducible elements.
The Davenport constant $\mathsf D (G)$ of $G$ is the maximal length of a minimal zero-sum sequence over $G$ (equivalently, $\mathsf D (G)$ is the smallest integer $\ell \in \N$ such that every sequence over $G$ of length at least $\ell$ has  a non-empty zero-sum subsequence).

In this note we study a conjecture stemming from factorization theory. We first formulate it in basic terms. Its background and significance will be discussed in Section \ref{2}, when we have more terminology at our disposal (Theorem \ref{2.2} and Corollary \ref{2.3}).

\smallskip
\noindent
\begin{conjecture} \label{1.1}
 Let $G$ be a finite abelian group, which is neither cyclic nor an elementary $2$-group.  Then, for every minimal zero-sum sequence $U = g_1 \ldots g_{\ell}$ of length $|U| = \ell = \mathsf D (G)$,  there are $k \in \N$ and minimal zero-sum sequences $U_1, \ldots, U_k, V_1, \ldots, V_{k+1}$ with terms from $\{g_1, \ldots, g_{\ell}, -g_1, \ldots, - g_{\ell} \}$ such that $U_1 \ldots U_k = V_1 \ldots V_{k+1}$.
\end{conjecture}

\smallskip
Let $G$ be a cyclic group of order $|G|=n \ge 3$. Then $\mathsf D (G)=n$ and every minimal zero-sum sequence over $G$ of length $n$ consists of an element $g$ of order $n$ repeated $n$ times. Thus all distances $s-r$, occurring in equations $U_1 \ldots U_r = V_1 \ldots V_s$ over minimal zero-sum sequences with terms from $\{-g, g\}$, is a multiple of $n-2$. Similarly, if $G$ is an elementary $2$-group of rank $r \ge 2$, then $\mathsf D (G)=r+1$ and all distances $s-r$ are multiplies of $r-1$. Thus,  the above conjecture neither holds for cyclic groups  nor for elementary $2$-groups with Davenport constant greater than or equal to four.

\smallskip
To describe the challenge of the above conjecture, suppose that $G \cong C_{n_1} \oplus \ldots \oplus C_{n_r}$, where $r = \mathsf r (G) = \max \{\mathsf r_p (G) \colon p \in \mathbb P \}$ is the rank of $G$, $\mathsf r_p (G)$ is the $p$-rank of $G$ for every prime $p$,  and $1 < n_1 \mid \ldots \mid n_r$ are positive integers. Then
\begin{equation} \label{davenport}
\mathsf D^* (G) := 1 + \sum_{i=1}^r (n_i-1) \le \mathsf D (G) \,.
\end{equation}
It is known since the 1960s that equality holds for $p$-groups and for groups of rank $\mathsf r (G) \le 2$. There are further sparse series of groups where equality holds and groups where equality does not hold (see \cite{Bh-SP07a, Gi18a, Gi-Sc19a, Li20a} for recent progress). Even less is known for the associated inverse question asking for the structure of minimal zero-sum sequences of length $\mathsf D (G)$. However,  the full structural description of minimal zero-sum sequences of length $\mathsf D (G)$ is not always needed in order to settle the Conjecture \ref{1.1}. We summarize what is known so far.

\smallskip
Conjecture \ref{1.1} is settled for groups of rank two and for groups isomorphic to $C_2 \oplus C_2 \oplus C_{2n}$ with $n \ge 2$. These proofs heavily depend on the complete structural description of minimal zero-sum sequences of length $\mathsf D (G)$. Furthermore, the conjecture is proved for
groups isomorphic to $C_{p^k}^r$, where $p$ is a prime and $k, r \in \N$ such that $p^k > 2$, although for these groups there is not even a conjecture concerning the structure of minimal zero-sum sequences of maximal length (for all these results see \cite{Ge-Zh18a}). We formulate a main result of the present paper.

\begin{theorem} \label{1.2}
Conjecture \ref{1.1} holds true for the following non-cyclic finite abelian groups $G$.
\begin{itemize}
\item[(a)]  $G$ is a $p$-group such that $\gcd (\exp (G)-2, \mathsf D (G)-2)=1$.

\item[(b)]  $G\cong C_{p^{s_1}}^{r_1}\oplus C_{p^{s_2}}^{r_2}$, where $p$ is a prime and $r_1,r_2, s_1, s_2\in \N$ such that $s_1$ divides $s_2$.

\item[(c)] $G$ is a  group with exponent $\exp (G) = pq$, where $p, q$ are distinct primes satisfying one of the three properties.
           \begin{itemize}
           \item[(i)] $\gcd ( pq-2, \mathsf D (G)-2 ) = 1$.
           \item[(ii)] $\gcd(pq - 2, p+q-3)=1$.
           \item[(iii)] $q=2$ and $p-1$ is a power of $2$.
           \item[(iv)] $q=2$ and $\mathsf r_p (G) = 1$.
           \end{itemize}

\item[(d)]  $G$ is a  group with exponent $\exp (G) \in [3,11]\setminus \{8\}$.
\end{itemize}
\end{theorem}

Since Conjecture \ref{1.1} does not hold for groups $G$ with $\exp (G)=2$, groups that are sums of two elementary $p$-groups (as listed in (c)) and groups with small exponents, as listed in (d), are extremal cases for the validity of the conjecture. Statement (a) has a simple proof. However, since for $p$-groups we have $\mathsf D (G) = \mathsf D^* (G)$, it yields a variety of groups satisfying the conjecture. The precise value of the Davenport constant is not known in general for groups with exponent $\exp (G)=pq$, where $p$ and $q$ are distinct primes. To mention a few examples of what is known so far, let
 $G \cong C_2^r \oplus C_6$ with $r \in \N$. Then $\mathsf D (G) = \mathsf D^* (G)$ (i.e., equality holds in \eqref{davenport}) if and only if $r \in [1,3]$ (see \cite[Corollary 2]{Ge-Sc92} and \cite{Sa-Ch14a}). Moreover, if a group $G$ with $\exp (G)=6$ has a subgroup isomorphic to $C_2^i \oplus C_6^{5-i}$ for some $i \in [1,4]$, then $\mathsf D (G) > \mathsf D^* (G)$ by \cite[Theorem 3.1]{Ge-Li-Ph12}.

We proceed as follows. In Section \ref{2}, we present the background from factorization theory which motivates the above conjecture. We formulate a conjecture and a theorem in terms of factorization theory (Conjecture \ref{2.1} and Theorem \ref{2.2}), associated to the ones given in the Introduction. In Corollary \ref{2.3}, we establish the significance of the two conjectures for the structure of sets of lengths having maximal elasticity. In Section \ref{3}, we prove Theorems \ref{1.2} and \ref{2.2}.

\section{Background on sets of lengths}\label{2}
\smallskip

For integers $a, b \in \Z$, we denote by $[a, b] = \{ x \in \Z \colon a \le x \le b\}$ the discrete interval between $a$ and $b$. Let $L = \{m_1, \ldots, m_k\} \subset \Z$ be a finite nonempty subset with $k \in \N$ and $m_1 < \ldots < m_k$. Then $\Delta (L) = \{m_i - m_{i-1} \colon i \in [2,k] \} \subset \N$ denotes the set of distances of $L$. If $L' \subset \Z$ is a finite subset, then $L+L' = \{a + a' \colon a \in L, a' \in L'\}$ is the sumset of $L$ and $L'$. If $L \subset \N$ consists of positive integers, then $\rho (L) = \max L / \min L$ denotes the elasticity of $L$ and for convenience we set $\rho ( \{0\}) = 1$. Let $G$ be an additively written finite abelian group. If $G_0 \subset G$ is a subset, then $\langle G_0 \rangle$ is the subgroup generated by $G_0$. Let $r \in \N$ and $(e_1, \ldots, e_r)$ be an $r$-tuple of elements of $G$. Then $(e_1, \ldots, e_r)$ is said to be independent if $e_i \ne 0$ for all $i \in [1,r]$ and if for all $m_1, \ldots, m_r \in \Z^r$ an equation $m_1e_1 + \ldots + m_re_r = 0$ implies that $m_ie_i=0$ for all $i \in [1,r]$. Furthermore, $(e_1, \ldots, e_r)$ is a basis of $G$ if it is independent and $G = \langle e_1 \rangle \oplus \ldots \oplus \langle e_r \rangle$. A subset $G_0 \subset G$ is independent if the tuple $(g)_{g \in G_0}$ is independent.
We recall some basics of the arithmetic of monoids and of zero-sum sequences. Our notation and terminology are consistent with \cite{Ge-Ru09, Gr13a}.

\smallskip
\noindent
{\bf Arithmetic of Monoids.} By a monoid, we mean a commutative cancellative semigroup with identity element. Let $H$ be a multiplicatively written monoid. We denote by $\mathcal A (H)$ the set of atoms (irreducible elements) of $H$ and say that $H$ is {\it atomic} if every non-invertible element can be written as a finite product of atoms. If $a = u_1 \ldots u_k$, where $k \in \N$ and $u_1, \ldots, u_k \in \mathcal A (H)$, then $k$ is a factorization length of $a$, and
\[
\mathsf L_H (a) = \mathsf L (a) = \{k \colon k \ \text{is a factorization length of} \ a \} \subset \N
\]
denotes the {\it set of lengths} of $a$. It is usual to set $\mathsf L (a) = \{0\}$ if $a \in H$ is invertible. The family
\[
\mathcal L (H) = \{\mathsf L (a) \colon a \in H \}
\]
is called the {\it system of sets of lengths} of $H$ and
\[
\rho (H) = \sup \{ \rho (L) \colon L \in \mathcal L (H) \} \in \R_{\ge 1} \cup \{\infty\}
\]
denotes the {\it elasticity} of $H$. Furthermore,
\[
\Delta (H) = \bigcup_{L \in \mathcal L (H)} \Delta (L) \ \subset \N
\]
is the {\it set of distances} of $H$. By definition, $\rho (H)=1$ if and only if $\Delta (H)=\emptyset$, and otherwise we have $\min \Delta (H) = \gcd \Delta (H)$.

\smallskip
\noindent
{\bf Zero-sum Sequences.} Let $G$ be an additively written finite abelian group and $G_0 \subset G$ be a subset. We denote by
$\mathcal F (G_0)$ the (multiplicatively written) free abelian monoid with basis $G_0$, called the {\it monoid of sequences} over $G_0$.
Let
\[
S = g_1 \ldots g_{\ell} = \prod_{g \in G_0} g^{\mathsf v_g (S)} \in \mathcal F (G_0)
\]
be a sequence over $G_0$. Then, for every $g \in G_0$,  $\mathsf v_g (S) \in \N_0$ is the multiplicity of $g$ in $S$, $\supp (S) = \{g_1, \ldots, g_{\ell} \} \subset G_0$ is the support of $S$,  $|S|=\ell = \sum_{g \in G_0} \mathsf v_g (S) \in \N_0$ is the length of $S$, $\Sigma (S) = \{\sum_{i \in I} g_i \colon \emptyset \ne I \subset [1, \ell] \}$ is the set of subsequence sums of $S$, and $\sigma (S) = g_1+ \ldots + g_{\ell} \in G$ is the sum of $S$. We say that $S$ is zero-sum free if $0 \notin \Sigma (S)$. The set
\[
\mathcal B (G_0) = \{ S \in \mathcal F (G_0) \colon \sigma (S) = 0 \} \subset \mathcal F (G_0)
\]
is a submonoid of $\mathcal F (G_0)$, called the {\it monoid of zero-sum sequences} over $G_0$. We set
\[
\mathcal L (G_0) := \mathcal L ( \mathcal B (G_0)), \ \Delta (G_0) := \Delta ( \mathcal B (G_0)),  \ \rho (G_0) := \rho ( \mathcal B (G_0)) \,,
\]
and so on.

\smallskip
\noindent
{\bf Transfer Krull monoids.} A monoid $H$ (resp. a domain $D$) is said to be a {\it transfer Krull monoid} (resp. a transfer Krull domain) over a finite abelian group $G$ if there exists a transfer homomorphism $\theta \colon H \to \mathcal B (G)$ (resp. $\theta \colon D \setminus \{0\} \to \mathcal B (G)$). The classical example of a transfer Krull domain is the ring of integers $\mathcal O_K$ of an algebraic number field $K$, and in this case $G$ is the ideal class group of $\mathcal O_K$. We refer to the survey \cite{Ge-Zh20a} for formal definitions and further examples. The crucial property of a transfer homomorphism $\theta \colon H \to \mathcal B (G)$ is that it preserves the system of sets of lengths. We have $\mathcal L (H) = \mathcal L (G)$, whence all invariants describing the structure of sets of length coincide. In particular, we have
\begin{equation} \label{upper-bounds}
\Delta (H) = \Delta (G) \subset [1, \mathsf D (G)-2] \quad \text{and} \quad \rho (H) = \rho (G) = \mathsf D (G)/2 \,.
\end{equation}
We refer to the survey \cite{Sc16a} for what is known on the system $\mathcal L (G)$ and on associated invariants. Let $H$ be a transfer Krull monoid over $G$. It is classical that $|L|=1$ for all $L \in \mathcal L (H)$ if and only if $|G| \le 2$. Suppose that $|G| \ge 3$. Then there is $a \in H$ such that $|\mathsf L (a)| > 1$. For every $n \in \N$,  the $n$-fold sumset
\[
\mathsf L (a) + \ldots + \mathsf L (a) \subset \mathsf L (a^n) \,,
\]
whence $|\mathsf L (a^n)| > n$. Thus, sets of lengths in $\mathcal L (H)$ can be arbitrarily large.

\smallskip
\noindent
{\bf On $\Delta_{\rho} (H)$.}
Now we define the crucial invariant of the present paper (see \cite[Definition 2.1]{Ge-Zh18a}). Let $\Delta_{\rho} (H)$ denote the set of all $d \in \N$ with the following property: for every $k \in \N$, there is some $L_k \in \mathcal L (H)$ with $\rho (L_k) = \rho (H)$ and which has the  form
\begin{equation} \label{eq:defAAP}
L_k = y + (L' \cup \{0, d , \ldots, \ell d\} \cup L'') \subset y + d \mathbb Z
\end{equation}
where $y\in \Z$, $\ell \ge k$,  $\max L' < 0$, and $\min L'' > \ell D$. If $H$ is a transfer Krull monoid over a finite abelian group $G$, then $\Delta_{\rho} (H) = \Delta_{\rho} (G)$ and there is a constant $M \in \N_0$ such that $L' \subset [-M,-1]$, and $L'' \subset \ell d + [1, M]$ (\cite[Lemma 2.3]{Ge-Zh18a}). The following conjecture was first formulated in \cite[Conjecture 3.20]{Ge-Zh18a}.

\smallskip
\begin{conjecture} \label{2.1}
Let $H$ be a transfer Krull monoid over a finite abelian group $G$ with $|G| > 4$. Then $\Delta_{\rho} (H)= \{1\}$ if and only if $G$ is neither cyclic nor an elementary $2$-group.
\end{conjecture}

\smallskip
In the present note we study $\Delta_{\rho} (G)$ and obtain the following result.

\smallskip
\begin{theorem} \label{2.2}
Let $H$ be a transfer Krull monoid over a finite abelian non-cyclic group $G$. Then $\Delta_{\rho} (H) = \{1\}$ for the following groups.
\begin{itemize}
	\item[(a)]  $G$ is a $p$-group such that $\gcd (\exp (G)-2, \mathsf D (G)-2)=1$.
	
	\item[(b)]  $G\cong C_{p^{s_1}}^{r_1}\oplus C_{p^{s_2}}^{r_2}$, where $p$ is a prime and $r_1,r_2, s_1, s_2\in \N$ such that $s_1$ divides $s_2$.
	
	\item[(c)] $G$ is a  group with exponent $\exp (G) = pq$, where $p, q$ are distinct primes satisfying one of the three properties.
	\begin{itemize}
		\item[(i)] $\gcd ( pq-2, \mathsf D (G)-2 ) = 1$.
		\item[(ii)] $\gcd(pq - 2, p+q-3)=1$.
		\item[(iii)] $q=2$ and $p-1$ is a power of $2$.
		\item[(iv)] $q=2$ and $\mathsf r_p (G) = 1$.
	\end{itemize}

	\item[(d)]  $G$ is a  group with exponent $\exp (G) \in [3,11]\setminus \{8\}$.
\end{itemize}

\end{theorem}

The proof of Theorem \ref{2.2} will be given in Section \ref{3}. We derive  a corollary, which demonstrates the significance of the Conjecture \ref{2.1} and of Theorem \ref{2.2}. It states that, if $\Delta_{\rho} (H) = \{1\}$, then all sets of lengths $L$ with maximal elasticity $\rho (L) = \rho (H)$ are intervals, apart from their globally bounded initial and end parts.

\smallskip
\begin{corollary} \label{2.3}
Let $H$ be a transfer Krull monoid over a finite abelian group $G$ and suppose that $\Delta_{\rho} (H) = \{1\}$. Then there exists a constant $M^* \in \N_0$ such that every $L \in \mathcal L (H)$ with $\rho (L)= \rho (H)$ has the form
\[
L = y + (L' \cup [0, \ell]  \cup L'')  \,,
\]
where $y \in \Z$, $\ell \in \N_0$, $L' \subset [-M^*,-1]$, and $L'' \subset \ell  + [1, M^*]$.
\end{corollary}

\begin{proof}
Since $\mathcal L (H) = \mathcal L (G)$, it is sufficient to prove the claim for the monoid $\mathcal B (G)$ of zero-sum sequences over $G$. If $\mathsf D (G) \le 3$, then $\Delta (G) \subset \{1\}$, whence all $L \in \mathcal L (G)$ are intervals and the claim holds with $M^* = 0$. Suppose that $\mathsf D (G) \ge 4$ and recall that $\Delta(G) \subset [1, \mathsf D (G)-2]$ (see \eqref{upper-bounds}). We proceed in four steps.

\smallskip
\noindent
{\bf 1.} By \cite[Section 4.7]{Ge-HK06a}, there is a constant $M_1 \in \N_0$  such that every $L \in \mathcal L (G)$ has the form
\begin{equation} \label{representation-1}
L = y + (L' \cup L^* \cup L'') \subset y + (\mathcal D + d \Z) \,,
\end{equation}
where $y \in \Z$ is a shift parameter,
\begin{itemize}
\item $d \in \Delta (G) \subset [1, \mathsf D (G)-2]$ and  $\{0, d \} \subset \mathcal D \subset [0, d]$,

\item $L^*$ is finite nonempty with $\min L^* = 0$ and $L^* = (\mathcal D + d \Z) \cap [0, \max L^*]$,

\item $L' \subset [-M_1, -1]$, and $L'' \subset \max L^* + [1, M_1]$.
\end{itemize}
As a side remark, we recall that the above description is best possible, as it was shown by a realization result of Schmid (\cite{Sc09a}).

\smallskip
\noindent
{\bf 2.} Let $G_0 \subset G$ be a subset with $\Delta (G_0) \ne \emptyset$. By \cite[Theorem 4.3.6]{Ge-HK06a} (applied to the monoid $\mathcal B (G_0)$),  there are constants $\psi (G_0)$ and $ M_2 (G_0) \in \N_0$ such that for every $A \in \mathcal B (G_0)$ with $\mathsf v_g (A) \ge \psi (G_0)$ for all $g \in G_0$,
\begin{equation} \label{representation-2}
\mathsf L (A) = y_A + (L_A' \cup \{0, d_A , 2d_A, \ldots, s_A d_A \} \cup L_A'') \subset y_A + d_A \Z \,,
\end{equation}
where $y_A \in \Z$, $d_A = \min \Delta (G_0)$, $s_A \ge M_1+\mathsf D (G)$, $L_A' \subset [-M_2 (G_0), -1]$, and $L_2'' \subset s_A d_A + [1, M_2 (G_0)]$. Since $G$ has only finitely many subsets $G_0$ with $\Delta (G_0) \ne \emptyset$, we let $\psi$ be the maximum over all $\psi (G_0)$ and let $M_2$ be the maximum over all $M_2 (G_0)$. Then the structural statement \eqref{representation-2} holds with  constants $\psi$ and $M_2$ for all subsets $G_0 \subset G$ with $\Delta (G_0) \ne \emptyset$.

\smallskip
\noindent
{\bf 3.} Clearly, it is sufficient to prove the claim of the corollary for all $A \in \mathcal B (G)$ with $\rho ( \mathsf L (A)) = \mathsf D (G)/2$, for which $\max \mathsf L (A) - \min \mathsf L (A)$ is sufficiently large. Indeed, suppose that there are constants $M_3, M_4 \in \N_0$ such that the claim holds for all $A$ with $\mathsf L (A) \not\subset \min \mathsf L (A) + [0, M_3]$ and with bound $M_4$ for the initial and end parts of $\mathsf L (A)$. Then the claim holds for all $A$ with bound $\max \{M_3, M_4\}$ for the initial and end parts of $\mathsf L (A)$.

\smallskip
\noindent
{\bf 4.} Now let $A \in \mathcal B (G)$ with $\rho ( \mathsf L (A)) = \mathsf D (G)/2$. By \cite[Lemma 3.2.(a)]{Ge-Zh18a}, there are $k, \ell \in \N$ and $U_1, \ldots, U_k, V_1, \ldots, V_{\ell} \in \mathcal A (G)$ with $|U_1|= \ldots = |U_k|= \mathsf D (G)$ and $|V_1| = \ldots = |V_{\ell}|=2$ such that $A = U_1 \ldots U_k = V_1 \ldots V_{\ell}$.  Then $k = \min \mathsf L (A)$ and $\ell = \max \mathsf L (A) = k \mathsf D (G)/2$. By {\bf 3.}, we may suppose that $k \ge |\mathcal A (G)| \psi$. Then there is $i \in [1,k]$, say $U_i = U$, such that  $U^{\psi}$ divides $A$. This implies that $(-U)^{\psi}U^{\psi}$ divides $A$, say $A = (-U)^{\psi}U^{\psi}B_{\psi}$ for some $B_{\psi} \in \mathcal B (G)$.
By {\bf 2.} (applied to the subset $\supp \big( (-U)U \big)$),
\begin{equation} \label{representation-3}
\mathsf L ( (-U)^{\psi}U^{\psi}) = y_U + (L_U' \cup \{0,d_U , 2 d_U, \ldots, s_U d_U \} \cup L_U'') \subset y_U + d_U \Z \,,
\end{equation}
where $y_{U} \in \Z$, $s_U \in \N$ with $s_U \ge M_1+\mathsf D (G)$, $d_U = \min \Delta (G_U)$, $L_U' \subset [-M_2, -1]$, and $L_U'' \subset s_U d_U + [1, M_2]$. Since $\Delta_{\rho} (G) = \{1\}$, \cite[Corollary 3.3]{Ge-Zh18a} implies that $d_U=1$. Since
\[
\mathsf L (B_{\psi}) + \mathsf L ( (-U)^{\psi} U^{\psi}) \subset \mathsf L (A) \,,
\]
$\mathsf L (A)$ contains an interval $[t, t+s_U]$ for some $t \in \N_0$.
By {\bf 3.}, we may assume that $\mathsf L (A)$ is not contained in $\min \mathsf L (A) +  [0, 2M_1+\mathsf D (G)]$. Thus, by comparing the two representations \eqref{representation-1} and \eqref{representation-3}, we infer that the period  $\mathcal D$ in \eqref{representation-1} is an interval. Thus, $L^*$ is an interval, whence $\mathsf L (A)$ has the required form.
\end{proof}

\begin{remark} \label{2.4}
Let $G$ be a finite abelian group. If $\Delta (G) = \{1\}$ (which, for example, holds if $G \cong C_3 \oplus C_3$), then all sets of lengths are intervals. In particular, Corollary \ref{2.3} holds with $M^* = 0$. Suppose that $G = C_p^r$ is an elementary $p$-group with $p \ge 5$ and $r \ge 2$.

1. Let $(e_1, \ldots, e_r)$ be a basis of $G$ and $e_0 = e_1 + \ldots + e_r$. Then $U = e_1^{p-1} \ldots e_r^{p-1} e_0 \in \mathcal A (G)$ with $|U|=\mathsf D (G)$. For  every $k \in \N$, we set $A_k = (-U)^k U^k$. Then $\rho ( \mathsf L (A_k)) = \mathsf D (G)/2$ and $\min \mathsf L (A_k)=2k$. It is easy to see that $2k+1 \notin \mathsf L (A_k)$, whence the constant $M^*$, occurring in Corollary \ref{2.3}, cannot be zero but is strictly positive.

2. Every nonzero element $g \in G$ can be extended to a basis. Thus, every nonzero element of $G$ occurs in the support of a minimal zero-sum sequence of length $\mathsf D (G)$. Therefore, for every $k \in \N$, there is $B_k \in \mathcal B (G)$ with $\rho ( \mathsf L (B_k))= \mathsf D (G)/2$, $\supp (B_k) = G \setminus \{0\}$, and $\min \mathsf L (B_k) \ge k$. Since $\mathsf L (B)$ is an interval for all $B \in \mathcal B (G)$ with $\supp (B)= G \setminus \{0\}$ (\cite[Theorem 7.6.9]{Ge-HK06a}), all sets $\mathsf L (B_k)$ are intervals with elasticity $\mathsf D (G)/2$.
\end{remark}

\section{Proof of  Theorems \ref{1.2} and \ref{2.2}} \label{3}

In this section, we prove Theorem \ref{1.2} and Theorem \ref{2.2}. We start with two lemmas.

\begin{lemma} \label{3.1}
Let $G$ be a finite abelian group with rank $\mathsf r(G)\ge 2$ and $\exp(G)\ge 3$, and let $U\in \mathcal A(G)$  with  $|U|=\mathsf D(G)$. If  there exist an independent tuple $(e_1,\ldots,e_t) \in G^t$ with $t\ge 2$ and an element $g$ such that $\{e_1,\ldots, e_t,g\} \subset \supp(U)$ and $ag=k_1e_1+\ldots+k_te_t$ for some $a\in [1,\ord(g)-1]\setminus\{\frac{\ord(g)}{2}\}$ and with $k_i\in [1, \ord(e_i)-1]$ for all $i\in [1,t]$,  then $\min \Delta \big( \supp \big( (-U)U \big) \big)=1$.  In particular, if $\supp(U)$ contains a basis of $G$, then $\min \Delta \big( \supp \big( (-U)U \big) \big)=1$.
\end{lemma}

\begin{proof}
See \cite[Lemma 3.10]{Ge-Zh18a}.
\end{proof}

\begin{lemma} \label{3.2}
Let $G$ be a finite abelian group such that $G \cong C_{p^{s_1}}^{r_1} \oplus C_{p^{s_2}}^{r_2}$, where $p$ is a prime, $r_1, r_2, s_1,s_2 \in \N$ with $s_1<s_2$, and let $G_0 \subset G$ be a subset with $\langle G_0\rangle =G$. Then there is a subset $G_0' \subset G_0$ such that $\langle G_0' \rangle \cong C_{p^{s_2}}^{r_2}$ and $G_0'$ is a basis of $\langle G_0' \rangle$.
\end{lemma}

\begin{proof}
Let $G_1$ and $G_2$ be subgroups of $G$ such that $G=G_1\oplus G_2$,  $G_1 \cong C_{p^{s_1}}^{r_1} $, and $G_2\cong C_{p^{s_2}}^{r_2}$. Then every element $g\in G_0$ can be written uniquely as $g=u_g+v_g$, where $u_g\in G_1$ and $v_g\in G_2$. Hence $\langle v_g\colon g\in G_0\rangle=G_2$ and $\{v_g\colon  g\in G_0\}$ contains a basis of $G_2$ by \cite[Lemma A.7.3]{Ge-HK06a}. We choose elements $g_1, \ldots, g_{r_2}\in G_0$ such that $(v_1=v_{g_1}, \ldots, v_{r_2}=v_{g_{r_2}})$ is a basis of $G_2$. Note that $\ord(v_i)=\ord(g_i)=p^{s_2}$ for every $i\in [1,r_2]$.
If $k_1,\ldots, k_{r_2}\in [0, p^{s_2}-1]$ such that $k_1g_1+\ldots +k_{r_2}g_{r_2}=0$, then $k_1v_1+\ldots+k_{r_2}v_{r_2}=0$,  whence the independence of $(v_1,\ldots, v_{r_2})$ implies that $k_1=\ldots=k_{r_2}=0$. It follows that $(g_1,\ldots, g_{r_2})$ is independent and hence $\langle g_1,\ldots, g_{r_2}\rangle \cong C_{p^{s_2}}^{r_2}$.
\end{proof}

\begin{proof}[Proof of Theorems \ref{1.2} and \ref{2.2}]
Let $H$  be a monoid, $G$ be a finite abelian non-cyclic group, and let $\theta \colon H \to \mathcal B (G)$ be a transfer homomorphism. Then $\Delta_{\rho} (H) = \Delta_{\rho} (G)$. In order to show that $\Delta_{\rho} (G) = \{1\}$, it is sufficient to show that
\begin{equation} \label{basic}
\min\Delta \big(\supp \big((-U)U \big) \big)=1 \ \text{ for every minimal zero-sum sequence $U$ over $G$ with $|U|=\mathsf D(G)$}
\end{equation}
(see \cite[Corollary 3.3.2]{Ge-Zh18a}). Note that \eqref{basic} is precisely the statement of Conjecture \ref{1.1}. Let $U$ be a minimal zero-sum sequence over $G$ with $|U|= \mathsf D (G)$. We set $d = \min\Delta \big(\supp \big((-U)U \big) \big)$ and have to show that $d=1$. Since $|U| = \mathsf D (G)$, we have $G = \langle \supp (U) \rangle$ by \cite[Proposition 5.1.4]{Ge-HK06a}.

Let $A\subset \supp(U)$ be a minimal subset such for every element $g\in \supp(U)\setminus A$, there exists $h\in A$ such that $g\in \langle h\rangle$.
Thus, for any two elements $g_1,g_2\in A$, we have $g_1\not\in \langle g_2\rangle$ and $\langle A\rangle = \langle \supp (U) \rangle =G$ is not cyclic, whence $|A|\ge 2$.  Assume to the contrary that $A$ is independent. We set  $A=\{g_1, \ldots, g_m \}$ and $W_i=\prod_{g\in \langle g_i\rangle }g^{\mathsf v_g(U)}$ for every $i\in [1,m]$, where $m = |A| \ge 2$. Then $U=\prod_{i\in [1,m]}W_i$ and $\sigma(W_i)\in \langle g_i\rangle$ for every $i\in [1,m]$. Since $A$ is independent and $U$ is a zero-sum sequence, we obtain that $W_i$ are zero-sum sequences for all $i\in [1,m]$, a contradiction to the minimality of $U$. Thus $A$ is not independent.

We start with two simple observations. Let $V$ be a minimal zero-sum sequence over $\supp \big(U(-U) \big)$.
Since $(-V)V$ has a factorization of length $|V|$, it follows that
\begin{equation} \label{observation-2}
d \quad  \text{divides} \quad  |V| -2  \,. \quad \text{ In particular, }\quad  d \quad  \text{divides} \quad \mathsf  D(G) -2 \,.
\end{equation}

If $g \in \supp \big( (-U)U \big)$ with $\ord (g)=n$, then $V=g^n$ is a minimal zero-sum sequence over $\supp \big(U(-U) \big)$, whence \eqref{observation-2} implies that $d \mid (n-2)$. Thus we obtain that
\begin{equation} \label{observation-1}
d \quad  \text{divides} \quad \ord (g)-2 \quad \text{for all $g \in \supp \big( (-U)U \big)$} \,.
\end{equation}

We distinguish four cases. Whenever it is convenient, an elementary $p$-group will be considered as a vector space over the field with $p$ elements.

\smallskip
\noindent
CASE 1: $G$ is a  $p$-group such that $\gcd ( \exp (G)-2, \mathsf D (G)-2)=1$.

  By \cite[Corollary 5.1.13]{Ge-HK06a}, $\supp (U)$ contains an element of order $\exp (G)$. Thus $d=1$ by  \eqref{observation-1} and \eqref{observation-2}.

\medskip
\noindent
CASE 2:  $G\cong C_{p^{s_1}}^{r_1}\oplus C_{p^{s_2}}^{r_2}$, where $p$ is a  prime and $r_1,r_2, s_1,s_2\in \N$ such that $s_1$ divides $s_2$.

If $s_1=s_2$, then the assertion follows from \cite[Theorem 3.11]{Ge-Zh18a}. Suppose that $s_1<s_2$. Then $\exp(G)=p^{s_2}\ge 4$.
By Lemma \ref{3.2}, there is a subset $A_2 \subset A$ such that $\langle A_2 \rangle \cong C_{p^{s_2}}^{r_2}$ and $A_2$ is a basis of $\langle A_2 \rangle$, say $A_2 = \{ g_1, \ldots, g_{r_2} \}$. Since $\langle A_2 \rangle$ is a direct summand of $G$, there is a subgroup
$G_1$  of $G$   with $G= G_1 \oplus \langle A_2\rangle$, whence $G_1\cong C_{p^{s_1}}^{r_1}$. Every element $g$ of $A$ can be written uniquely as $g=u_g+v_g$, where $u_g\in G_1$ and $v_g\in \langle A_2 \rangle$. Hence $\langle u_g\colon g\in A\rangle=G_1$ and $\{u_g\colon  g\in A\}$ contains a basis of $G_1$ by \cite[Lemma A.7.3]{Ge-HK06a}. We choose $h_1, \ldots, h_{r_1}\in A$ such that $(u_1=u_{h_1}, \ldots, u_{r_1}=u_{h_{r_1}})$ is a basis of $G_1$.
We distinguish two cases.

Suppose $\ord(h_i)=p^{s_1}$ for every $i\in [1, r_1]$.  Then the tuple $(h_1, \ldots, h_{r_1})$ is independent, whence the tuple $(h_1, \ldots, h_{r_1}, g_1, \ldots, g_{r_2})$ forms a basis of $G$. Then the assertion follows by Lemma \ref{3.1}.

Suppose there exists $i\in [1,r_1]$ such that $\ord(h_i)\neq p^{s_1}$. Then $0\neq p^{s_1}h_i$ and $p^{s_1}$ is the minimal integer such that $p^{s_1}h_i\in \langle g_1,\ldots, g_{r_2}\rangle$. Set $h=h_i$.
There exist  $\emptyset \ne I\subset [1,r_2]$ and $k_i\in [1, p^{s_1}-1]$ for every $i\in I$ such that $p^{s_1}h=\sum_{i\in I}k_ig_i$. After renumbering if necessary, we may assume that $I=[1,t]$ for some $t \in [1, r_2]$.

 If $t=1$, then $(-h)^{p^{s_1}}g_1^{k_1}$ and $h^{p^{s_1}}g_1^{p^{s_2}-k_1}$ are both minimal zero-sum sequences. It follows by $\big( h^{p^{s_1}}g_1^{p^{s_2}-k_1} \big) \big((-h)^{p^{s_1}}g_1^{k_1} \big)= \big(h(-h) \big)^{p^{s_1}}g_1^{p^{s_2}}$ that $d$ divides $p^{s_1}-1$. Since $d$ divides $p^{s_2}-2$ by \eqref{observation-1}, it follows by the fact that  $s_1$ divides $s_2$ that $d=1$.

  If $t\ge 2$ and $p^{s_1}\neq \frac{\ord(h)}{2}$, then $d=1$ by Lemma \ref{3.1}.

   If $t\ge 2$ and $\ord(h)=2p^{s_1}$, then $p=2$. Since $h^{2^{s_1}}g_1^{k_1}\ldots g_t^{k_t}$ is a minimal zero-sum sequence, we obtain that $g_1^{2k_1}\ldots g_t^{2k_t}$ is a zero-sum sequence, whence the independence of $(g_1,\ldots, g_t)$ implies that $k_i=2^{s_2-1}$ for all $i\in [1,t]$. Since $(-h)^{2^{s_1}}g_1^{k_1}\ldots g_t^{k_t}$ is a minimal zero-sum sequence and
\[
\big( h^{2^{s_1}}g_1^{k_1}\ldots g_t^{k_t} \big)^2 =h^{2^{s_1+1}}\ g_1^{2^{s_2}}\ldots g_t^{2^{s_2}}
\text{ and } \big(h^{2^{s_1}}g_1^{k_1}\ldots g_t^{k_t} \big)\ \big( (-h)^{2^{s_1}}g_1^{k_1}\ldots g_t^{k_t} \big)= \big(h(-h) \big)^{2^{s_1}}\ g_1^{2^{s_2}}\ldots g_t^{2^{s_2}}\,,
\]
we obtain that $d$ divides $(t+1-2)-(2^{s_1}+t-2)=1-2^{s_1}$. Since $d$ divides $2^{s_2}-2$ by \eqref{observation-1}, it follows by the fact that  $s_1$ divides $s_2$ that  $d=1$.

\medskip
\noindent
CASE 3:  $G$ is the sum of two elementary $p$-groups, say $G = C_p^r \oplus C_q^s$, where $p, q$ are distinct primes and $r, s \in \N$.

We set $U = U_pU_q U_{pq}$, where $U_p \in \mathcal F (G)$ consists of elements of order $p$, $U_q \in \mathcal F (G)$ consists of elements of order $q$, and $U_{pq} \in \mathcal F (G)$ consists of elements of order $pq$. Since
$U$ is a minimal zero-sum sequence  and $0 = \sigma (U) = \sigma (U_p) + \sigma (U_q) + \sigma (U_{pq})$, it follows that $U_{pq}$ cannot be the empty sequence. Thus, $\supp (U)$ contains an element of order $pq = \exp (G)$ and hence by \eqref{observation-1}, we have
\begin{equation}\label{equation2}
d \text{ divides }pq-2\,.
\end{equation}

\medskip
\noindent
CASE 3.(i): $\gcd (pq-2, \mathsf D (G)-2) = 1$.

Then $d=1$   by  \eqref{observation-2} and \eqref{equation2}.

\medskip
\noindent
CASE 3.(ii): $\gcd(pq-2, p+q-3)=1$.

We first note that $\gcd(pq-2, p+q-3)=1$ implies that $p,q$ are both odd, $\gcd(pq-2, p-2)=1$,  $\gcd(pq-2, q-2)=1$, and
 \begin{equation} \label{equation1}
\gcd(pq-2, p-1)=1\,.
 \end{equation}

	 If there exists an element $h\in \supp (U)$ such that $\ord(h)\neq pq$, then $d$ divides $p-2$ or $q-2$ by \eqref{observation-1}.  In each case we infer that  $d=1$ by \eqref{equation2}.
	Now, we  assume that every element of $\supp (U)$ has order $pq$.

	Let $G_1$ and $G_2$ be subgroups of $G$ such that $G=G_1\oplus G_2$,  $G_1 \cong C_p^{r} $, and $G_2\cong C_{q}^{s}$. Since $G$ is not cyclic, we have $r\ge 2$ or $s\ge 2$. By symmetry, we may assume that  $r\ge s$.  Every element $g$ of $A$ can be written uniquely as $g=u_g+v_g$, where $u_g\in G_1$ and $v_g\in G_2$. Hence $\langle u_g\colon g\in A\rangle=G_1$ and $\{u_g\colon  g\in A\}$ contains a basis of $G_1$. We choose $g_1, \ldots, g_{r}\in A$ such that $(u_1=u_{g_1}, \ldots, u_{r}=u_{g_{r}})$ is a basis of $G_1$. We distinguish two cases.
	
	First, suppose that $(v_1=v_{g_1},\ldots, v_r=v_{g_r}) \in G_2^r$ is  independent. Since $r\ge s = \mathsf r(G_2)=s$, we infer that $r=s$.
	Therefore, $G\cong C_{pq}^r$ and $(g_1,\ldots, g_r)$ is a basis of $G$. The assertion follows by Lemma \ref{3.1}.

	Now, suppose that $(v_1=v_{g_1},\ldots, v_r=v_{g_r})$ is  not independent. Since $\langle v_1, \ldots, v_r\rangle$ is a $q$-group, there exists $I\subset [1,r]$ such that $(v_i)_{i\in I}$ is a basis of $\langle v_1, \ldots, v_r\rangle$. After renumbering if necessary, we may assume that $I=[1,y]$, where $y\in [1,r-1]$. Then $(g_1, \ldots, g_y)$ is independent and $p$ is the minimal integer such that $pg_r\in \langle g_1,\ldots, g_y\rangle$. If there exists $i\in [1,y]$ such that $pg_r\in \langle g_i\rangle$, then there exists $k\in [1, pq-1]$ such that $g_r^pg_i^k$ and $(-g_r)^pg_i^{pq-k}$ are atoms, whence it follows by $(g_r^pg_i^k)\ ((-g_r)^pg_i^{pq-k})=(g_r(-g_r))^p\ g_i^{pq}$ that $d$ divides $p-1$. The assertion follows by \eqref{equation1}. Otherwise there exist $J\subset [1,y]$ with $|J|\ge 2$ and $k_j\in [1, pq-1]$ for every $j\in J$ such that $pg_r=\sum_{j\in J}k_jg_j$.
	Now the assertion follows by Lemma \ref{3.1}.

\medskip
\noindent
CASE 3.(iii):  $q=2$ and $p-1$ is a power of $2$.

If there is an element $g \in \supp(U)$ such that $\ord (g) = p$, then $d$ divides $\gcd(p-2, 2p-2)=1$ by \eqref{observation-1} and \eqref{equation2}, whence $d=1$.
Now, we suppose that
$\supp(U)$ contains no element of order $p$.  Since $d$ divides $2(p-1)$ and $p-1$ is a power of $2$, it suffices to prove that $d$ divides an odd number.

Let $G_1$ and $G_2$ be subgroups of $G$ such that $G=G_1\oplus G_2$,  $G_1 \cong C_p^r $, and $G_2\cong C_2^s$. Then every element $g$ of $\supp(U)$ can be written uniquely as $g=u_g+v_g$, where $u_g\in G_1$ and $v_g\in G_2$. Hence $\langle u_g\colon g\in \supp(U)\rangle=G_1$ and $\{u_g: g\in \supp(U)\}$ contains a basis of $G_1$. We choose $g_1, \ldots, g_r\in \supp(U)$ such that $(u_1=u_{g_1}, \ldots, u_r=u_{g_r})$ is a basis of $G_1$. Since $\supp(U)$ has no element of order $p$, it follows that $\ord(g_i)=2p$ for all $i\in [1,r]$.

We set
\[
T_0=\prod_{g\in \supp(U)\text{ with }\ord(g)=2}g \quad \text{ and } \quad T_i=\prod_{g\in \langle g_i\rangle \text{ with }\ord(g)=2p}g^{\mathsf v_{g}(U)} \ \text{ for all } \ i\in [1,r] \,.
\]
Assume to the contrary that $U=T_0T_1\ldots T_r$. Then $\sigma(T_i)$ has order $2$ for every $i\in [1,r]$. If $|T_i|\ge p+1=\mathsf D(C_p)+1$, then there exists a subsequence $T_i'$ of $T_i$ such that $1\le |T_i'|\le p$ and $\sigma(T_i')$ has order $2$, which implies that $T_i'$ or $T_iT_i'^{-1}$ is a nonempty zero-sum sequence, a contradiction to the minimality of $U$.	
Thus $|T_i|\le p$ for every $i\in [1,r]$. Since $T_0$ is zero-sum free, we infer that $|T_0|\le r+s$.
Hence $|U|\le pr+(r+s)=(p+1)r+s<\mathsf D^*(G)\le \mathsf D(G)$, a contradiction.

Therefore, $U \ne T_0T_1\ldots T_r$, whence there is an element $h\in \supp(U)\setminus \{g_1,\ldots, g_r\}$ such that $\ord(h)=2p$ and $h\not\in \langle g_i\rangle$ for any $i\in [1,r]$.
Let $I\subset [1,r]$ be a minimal subset such that $u_h\in \langle u_i\colon i\in I\rangle$.
After renumbering if necessary, we may assume that $I=[1,x]$, where $x\in [1,r]$.

Suppose $x=1$. Note that $h\not\in \langle g_1\rangle$ and $2h\in \langle g_1\rangle$. Then there exists $k\in [1, 2p-1]$ such that $h^2g_1^k$ is a minimal zero-sum sequence,  and then the same is true for $(-h)^2g_1^{2p-k}$. Since  $(h^2g_1^k) ((-h)^2g_1^{2p-k})=(h(-h))^2 g_1^{2p}$, we obtain $d=1$.

Suppose $x\ge 2$ and $v_h\in \langle v_1, \ldots, v_x \rangle $, where $v_i=v_{g_i}$ for all $i\in [1,x]$. For all $i \in [1,x]$, let $k_i\in [1,p-1]$ be such that  $u_h=k_1u_1+\ldots + k_xu_x$. The elements $v_h, v_1, \ldots, v_x$ have order $2$.	After renumbering if necessary, we may  assume that $v_hv_1\ldots v_y$ is a minimal zero-sum sequence over $G_2$ for some $y \in [1, x]$. Therefore, the tuple $(v_1,\ldots, v_y)$ is independent, whence  $(g_1,\ldots, g_y)$ is independent. If $y=x$, then $h\in \langle g_1,\ldots, g_x\rangle$ and
the assertion follows by Lemma \ref{3.1}. Suppose $y<x$. Then $h\not\in \langle g_1, \ldots, g_y\rangle$ and $p$ is the minimal integer such that $ph\in \langle g_1, \ldots, g_y\rangle$.
If $y$ is even,  then $h^pg_1^p\ldots g_y^p$ is a minimal zero-sum sequence of odd length, whence $d$ divides an odd number  by \eqref{observation-2}. Suppose $y$ is odd. We replace $g_i$ by $-g_i$, $u_i$ by $-u_i$, and $k_i$ by $p-k_i$, if necessary, in order to make $k_i$ to be odd for all $i\in [1,y]$ and $k_j$ to be even for all $j\in [y+1,x]$. Then
\[
W=(-h)g_1^{k_1}\ldots g_x^{k_x} \quad \text{ and } \quad V=h^{p-1}(-g_1)^{p-k_1}\ldots (-g_y)^{p-k_y}g_{y+1}^{k_{y+1}}\ldots g_x^{k_x}
\]
are  minimal zero-sum sequences over $\supp(U(-U))$,  $T=g_{y+1}^{pk_{y+1}}\ldots g_x^{pk_x}$  is a zero-sum sequence,
\begin{align*}
W^p=T \ ((-h)g_1\ldots g_y)^p\ \prod_{i=1}^y (g_i^{2p})^{\frac{k_i-1}{2}} \quad \text{ and } \quad V^p=T \ ((-h)^{2p})^{\frac{p-1}{2} }\prod_{i=1}^y((-g_i)^{2p})^{\frac{p-k_i}{2}}\,.
\end{align*}
If $\ell_0\in \mathsf L(T)$, then $d$ divides
\[
\left(\ell_0+\frac{p-1}{2}+\sum_{i\in [1,y]}\frac{p-k_i}{2}-p\right)-\left(\ell_0+1+\sum_{i\in [1,y]}\frac{k_i-1}{2}-p\right)    =\frac{p-1}{2}+\frac{p+1}{2}y-\sum_{i=1}^yk_i-1
\]
Since $y$ is odd and $k_i$ are odd for all $i\in [1,y]$, it follows that $\frac{p-1}{2}+\frac{p+1}{2}y-\sum_{i=1}^yk_i-1\equiv 1\pmod 2$, whence $d$ divides an odd number.

Suppose $x\ge 2$ and $v_h\not\in \langle v_1, \ldots, v_x \rangle $. Then $h\not\in \langle g_1,\ldots, g_x\rangle$ and $2h\in \langle g_1,\ldots, g_x\rangle$. Let $2u_h=k_1u_1+\ldots +k_xu_k$, where $k_i\in [1,p-1]$. We replace $g_i$ by $-g_i$, $u_i$ by $-u_i$, and $k_i$ by $p-k_i$, if necessary,  in order to make $k_i$ to be  even for all $i\in [1,x]$. If $(g_1, \ldots, g_x)$ is independent, then the assertion follows by Lemma \ref{3.1}.
Otherwise the tuple $(v_1=v_{g_1}, \ldots, v_x=v_{g_x})$ is not independent.  After renumbering if necessary, we may assume that $v_1\ldots v_y$ is a minimal zero-sum sequence over $G_2$, where $y\in [2, x]$, whence $g_1^p\ldots g_y^p$ is a minimal zero-sum sequence. If $y$ is odd, then $d$ divides an odd number by \eqref{observation-2}. Suppose $y$ is even. Then $W=(-h)^2g_1^{k_1}\ldots g_x^{k_x}$  and  $V=(-h)^2(-g_1)^{p-k_1}\ldots (-g_y)^{p-k_y}g_{y+1}^{k_{y+1}}\ldots g_x^{k_x}$ are  minimal zero-sum sequences, $T=(-h)^{2p}g_{y+1}^{pk_{y+1}}\ldots g_x^{pk_x}$  is a zero-sum sequence,
\begin{align*}
W^p=T \  \prod_{i=1}^y (g_i^{2p})^{\frac{k_i}{2}} \quad \text{ and } \quad V^p=T \ ((-g_1)\ldots (-g_y))^p\  \prod_{i=1}^y((-g_i)^{2p})^{\frac{p-k_i-1}{2}}\,.
\end{align*}
If $\ell_0\in \mathsf L(T)$, then $d$ divides
\[
\left(\ell_0+1+\sum_{i\in [1,y]}\frac{p-k_i-1}{2}-p\right)-\left(\ell_0+\sum_{i\in [1,y]}\frac{k_i}{2}-p\right)    =1+\frac{p-1}{2}y-\sum_{i=1}^yk_i\,.
\]
Since $y$ is even and $k_i$ are even for all $i\in [1,y]$, it follows that $1+\frac{p-1}{2}y-\sum_{i=1}^yk_i\equiv 1\pmod 2$, whence $d$ divides an odd number.

\medskip
\noindent
CASE 3.(iv):  $q=2$ and $r=1$.

Let $h\in \supp(U)$ such that $\ord(h)=2p$.  Since $\langle h\rangle $ is a direct summand of $G$, there is a subgroup $G_1$ of $G$ with $G\cong G_1\oplus \langle h\rangle$, whence $G_1\cong C_2^s$. Every element $g$ of $\supp(U)$ can be written uniquely as $g=u_g+v_g$, where $u_g\in G_1$ and $v_g\in \langle h \rangle$. Hence $\langle u_g\colon g\in \supp(U)\setminus \{h\}\rangle=G_1$ and $\{u_g\colon  g\in \supp(U)\setminus \{h\}\}$ contains a basis of $G_1$. We choose $g_1, \ldots, g_s\in \supp(U)\setminus\{h\}$ such that $(u_1=u_{g_1}, \ldots, u_{s}=u_{g_{s}})$ is a basis of $G_1$. We distinguish two cases.

Suppose $\ord(g_i)=2$ for every $i\in [1, s]$.  Then the tuple $(g_1, \ldots, g_{s})$ is independent, whence the tuple $(g_1, \ldots, g_{s}, h)$ forms a basis of $G$. Then the assertion follows by Lemma \ref{3.1}.

Suppose there exists $i\in [1,s]$ such that $\ord(g_i)=2p$. Then $2$ is the minimal integer such that $2g_i\in \langle h\rangle$, whence there exists $k\in [1, 2p-1]$ such that both $g_i^2h^k$ and $(-g_i)^2h^{2p-k}$ are minimal zero-sum sequences. Then $d=1$ because
\[
\big(g_i^2h^k \big)\ \big((-g_i)^2h^{2p-k} \big) = \big(g_i(-g_i) \big)^2 \ h^{2p}\,.
\]

\medskip
\noindent
CASE 4: $G$ is a  group with $\exp (G) \in [3,11]\setminus \{8\}$.

If $\exp (G)$ is  prime, then the claim follows from CASE 2 (with $s_1=s_2=1$). The case, when $\exp (G) \in \{4, 9\}$,  is also handled in CASE 2,  and the case $\exp (G) \in \{6, 10\}$ is handled in CASE 3.(iii).
\end{proof}

\providecommand{\bysame}{\leavevmode\hbox to3em{\hrulefill}\thinspace}
\providecommand{\MR}{\relax\ifhmode\unskip\space\fi MR }
\providecommand{\MRhref}[2]{%
  \href{http://www.ams.org/mathscinet-getitem?mr=#1}{#2}
}
\providecommand{\href}[2]{#2}


\begin{thebibliography}{10}

\bibitem{Bh-SP07a}
G.~Bhowmik and J.-C. Schlage-Puchta, \emph{Davenport's constant for groups of
  the form $\mathbb{Z}_3 \oplus \mathbb{Z}_3 \oplus \mathbb{Z}_{3d}$}, Additive
  Combinatorics (A.~Granville, M.B. Nathanson, and J.~Solymosi, eds.), CRM
  Proceedings and Lecture Notes, vol.~43, American {M}athematical {S}ociety,
  2007, pp.~307 -- 326.

\bibitem{Sa-Ch14a}
F.~Chen and S.~Savchev, \emph{Long minimal zero-sum sequences in the groups
  ${C}_2^{r-1} \oplus {C}_{2k}$}, Integers \textbf{14} (2014), Paper A23.

\bibitem{Ge-HK06a}
A.~Geroldinger and F.~Halter-Koch, \emph{Non-{U}nique {F}actorizations.
  {A}lgebraic, {C}ombinatorial and {A}nalytic {T}heory}, Pure and Applied
  Mathematics, vol. 278, Chapman \& Hall/CRC, 2006.

\bibitem{Ge-Li-Ph12}
A.~Geroldinger, M.~Liebmann, and A.~Philipp, \emph{On the {D}avenport constant
  and on the structure of extremal sequences}, Period. Math. Hung. \textbf{64}
  (2012), 213 -- 225.

\bibitem{Ge-Ru09}
A.~Geroldinger and I.~Ruzsa, \emph{Combinatorial {N}umber {T}heory and
  {A}dditive {G}roup {T}heory}, Advanced Courses in Mathematics - CRM
  Barcelona, Birkh{\"a}user, 2009.

\bibitem{Ge-Sc92}
A.~Geroldinger and R.~Schneider, \emph{On {D}avenport's constant}, J. Comb.
  Theory, Ser. A \textbf{61} (1992), 147 -- 152.

\bibitem{Ge-Zh18a}
A.~Geroldinger and Q.~Zhong, \emph{Long sets of lengths with maximal
  elasticity}, Can. J. Math. \textbf{70} (2018), 1284 -- 1318.

\bibitem{Ge-Zh20a}
\bysame, \emph{Factorization theory in commutative monoids}, Semigroup Forum
  \textbf{100} (2020), 22 -- 51.

\bibitem{Gi18a}
B.~Girard, \emph{An asymptotically tight bound for the {D}avenport constant},
  J. {Ec.} {P}olytech. Math. \textbf{5} (2018), 605 -- 611.

\bibitem{Gi-Sc19a}
B.~Girard and W.A. Schmid, \emph{Direct zero-sum problems for certain groups of
  rank three}, J. Number Theory \textbf{197} (2019), 297 -- 316.

\bibitem{Gr13a}
D.J. Grynkiewicz, \emph{Structural {A}dditive {T}heory}, Developments in
  Mathematics 30, Springer, Cham, 2013.

\bibitem{Li20a}
Chao Liu, \emph{On the lower bounds of {D}avenport constant}, J. Comb. Theory,
  Ser. A \textbf{171} (2020), {105162, 15pp}.

\bibitem{Sc09a}
W.A. Schmid, \emph{A realization theorem for sets of lengths}, J. Number Theory
  \textbf{129} (2009), 990 -- 999.

\bibitem{Sc16a}
\bysame, \emph{Some recent results and open problems on sets of lengths of
  {K}rull monoids with finite class group}, in Multiplicative {I}deal {T}heory
  and {F}actorization {T}heory, Springer, 2016, pp.~323 -- 352.

\end{thebibliography}
\end{document}